\documentclass[a4paper,oneside,12pt]{article}

\usepackage{amsmath,amsfonts,amscd,amssymb}
\usepackage{longtable,geometry}
\usepackage[english]{babel}
\usepackage[utf8]{inputenc}
\usepackage[active]{srcltx}
\usepackage[T1]{fontenc}
\usepackage{graphicx}
\usepackage{pstricks}
\usepackage{bbm}
\usepackage{mathtools}
\usepackage{hyperref}

\usepackage{MnSymbol}
\usepackage{stmaryrd}
\usepackage{nicefrac}
\usepackage{calrsfs}

\usepackage{xcolor}
\usepackage{framed}

\colorlet{shadecolor}{blue!15}

\usepackage{amsthm}
\newtheorem{theorem}{Theorem}

\newtheorem{lemma}[theorem]{Lemma}
\newtheorem{proposition}{Proposition}

\newtheorem{remark}[theorem]{Remark}

\newcommand{\calA}{\mathcal{A}}

\newcommand{\calC}{\mathcal{C}}

\newcommand{\calE}{\mathcal{E}}
\newcommand{\calF}{\mathcal{F}}

\newcommand{\calH}{\mathcal{H}}

\newcommand{\calK}{\mathcal{K}}

\newcommand{\calX}{\mathcal{X}}

\newcommand{\bbN}{\mathbb{N}}

\newcommand{\bbZ}{\mathbb{Z}}

\newcommand{\sfN}{\mathsf{N}}

\newcommand{\llb}{\llbracket}
\newcommand{\rrb}{\rrbracket}

\newcommand{\lr}[1][]{\xleftrightarrow{\: #1 \:\:}}

\numberwithin{equation}{section}

\newcommand{\rk}[1]{\bgroup\color{red}%
  \par\medskip\hrule\smallskip%
  \noindent\textbf{#1}%
  \par\smallskip\hrule\medskip\egroup}

\title{A new computation of the critical point for the planar random-cluster model with $q\ge1$}
\author{Hugo Duminil-Copin, Aran Raoufi, Vincent Tassion}
\date{\today}

\begin{document}
\maketitle
 
 \begin{abstract}
 We present a new computation of the critical value of the random-cluster model with cluster weight $q\ge 1$ on $\bbZ^2$. This provides an alternative approach to the result in \cite{BefDum12}. We believe that this approach has several advantages. First, most of the proof can easily be extended to other planar graphs with sufficient symmetries. Furthermore, it invokes RSW-type arguments which are not based on self-duality. And finally, it contains a new way of applying sharp threshold results which avoid the use of symmetric events and periodic boundary conditions.
 \end{abstract}
 \section{Introduction} 

 The random-cluster model is one of the most classical generalization of Bernoulli percolation and electrical networks. This model was introduced by Fortuin and Kasteleyn in \cite{ForKas72} and has since then been the object of intense study, both physically and mathematically.

A {\em percolation configuration} on a graph $G=(V_G,E_G)$ ($V_G$ is the vertex set and $E_G$ the edge set) is an element $\omega=(\omega_e:e\in E_G)$ in $\{0,1\}^{E_G}$. An edge $e$ is said to be {\em  open} (in $\omega$) if $\omega_e=1$, otherwise it is {\em closed}. A configuration $\omega$ can be seen as a subgraph of $G$ with vertex set $V_\omega:=V_G$ and edge set $E_\omega:=\{e\in E_G:\omega_e=1\}$.

Let $p\in[0,1]$ and $q\ge1$. We will work with the {\em random-cluster measure
  $\phi_{p,q}^1$ on the square lattice with weights $(p,q)$ and wired boundary
  conditions}. Let us briefly recall its definition here.

The {\em square lattice} $\bbZ^2$ is defined to be the graph with
\begin{align*}V_{\bbZ^2}&:=\big\{(x,y):x,y\in\bbZ\big\},\\
E_{\bbZ^2}&:=\big\{\{(x,y),(x',y')\}\subset \bbZ^2\text{ such that } |x-x'|+|y-y'|=1\big\}.\end{align*}
Let $G$ be a finite subgraph of $\bbZ^2$ and let  $\phi_{G,p,q}^1$ be the measure on percolation configurations $\omega$ on $G$ defined by 
$$\phi_{G,p,q}^1(\omega)=\frac{p^{|E_\omega|}\ (1-p)^{|E_G\setminus E_\omega|}\ q^{k_1(\omega)}}{Z^1(G,p,q)},$$
where $k_1(\omega)$ is the number of connected components of the percolation configuration $\overline \omega$ on $\bbZ^2$ defined by $\overline \omega_e=\omega_e$ if $e\in E_G$, and 1 otherwise, and $Z^1(G,p,q)$ is a normalizing constant to make the total mass of the measure equal to 1.
Then, $\phi_{p,q}^1$ is the probability measure on percolation configurations on $\bbZ^2$ defined as the weak limit of the 
$\phi_{G,p,q}^1$ as $G$ exhausts $\bbZ^2$. We refer to \cite{Gri06} for a justification that this limit exists.


The random-cluster model on $\bbZ^2$ is known to undergo a phase transition for
$q\ge1$. Let $\{0\leftrightarrow\infty\}$ be the event that $0$ is in an
infinite connected component of $\omega$. There exists $p_c=p_c(q)$ such that
$\phi_{p,q}^1[0\leftrightarrow\infty]$ is equal to $0$ for $p<p_c$, and is
strictly positive if $p>p_c$. We give a new proof of the following result, which
was originally obtained in \cite{BefDum12}. 
\begin{theorem}\label{main:corollary}
Let $q\ge1$, the critical value $p_c=p_c(q)$ is equal to $\sqrt q/(1+\sqrt q)$. Furthermore, for  $p<p_c$, there exists $c=c(p,q)>0$ such that for any $x\in\bbZ^2$,
\begin{equation}\label{eq:abcd}\phi_{p,q}^1[0\lr x]\le \exp[-c \|x\|],\end{equation}
where $\|(a,b)\|=\max\{|a|,|b|\}$, and $\{0\lr x\}$ denotes the event that there
exists a path from $0$ to $x$ in $\omega$. 
\end{theorem}
As in \cite{BefDum12}, our strategy is based on the study of the crossing
probabilities, but we use more generic arguments. We will first prove
Theorem~\ref{thm:main} below, then we will deduce Theorem~\ref{main:corollary}
by using some self-duality properties specific to the random-cluster model on
the square lattice.

Let $\calC_h(a,b)$ be the event that there exits a sequence of vertices $v_0,\dots,v_k$ in $\llb -a,a\rrb\times\llb -b,b\rrb$ such that $v_0\in \{-a\}\times\llb -b,b\rrb$, $v_k\in \{a\}\times\llb -b,b\rrb$, and for any $0\le i<k$ the vertex $v_i$ is a neighbor of the vertex $v_{i+1}$ and $\omega_{v_i,v_{i+1}}=1$.
This event corresponds to the existence of a ``crossing from left to right'' in
the box $\llb -a,a\rrb\times\llb -b,b\rrb$.
\begin{theorem}\label{thm:main}
Let $q\ge1$ and $p\in[0,1]$. If 
\begin{equation}\label{eq:a}\inf_{n\ge 1}\phi_{p,q}^1[\calC_h(n,n)]>0,\tag{{$\mathcal A$}}\end{equation}
then for any $\delta\in(0,1-p]$, there exists $c>0$ such that for every $n\ge0$,
\begin{equation*}\label{eq:d}\phi^1_{p+\delta,q}[\calC_h(2n, n)] \ge 1-e^{-cn}.\end{equation*}
\end{theorem}

\begin{remark}
If $\phi_{p,q}^1[0\leftrightarrow\infty] > 0$, then \eqref{eq:a} holds: because of the symmetry of the lattice 
$$ \phi_{p,q}^1\left[0 \lr \{-n\} \times \llbracket -n,n \rrbracket\right] > \phi_{p,q}^1\left[ 0 \lr \infty \right] /4, $$ 
and the same holds for $ \phi_{p,q}^1\left[0 \lr \{n\} \times \llbracket -n,n \rrbracket\right]$. FKG inequality (see the next section) then implies \eqref{eq:a}.
\end{remark}

We isolated Theorem~\ref{thm:main} because its proof does not involve duality
arguments. Therefore, it is valid for any planar lattice with sufficient
symmetries. By a duality argument presented in
Section~\ref{sec:proof-coroll-refm}, it is sufficient to compute the critical
value and prove that the phase transition is sharp on the square lattice. This
duality argument is not valid if we only assume the symmetries necessary for
Theorem~\ref{thm:main}, and we refer to \cite{duminil2014phase} for a proof of sharpness
of the phase transition for models with such symmetries.

\bigbreak The proof of Theorem~\ref{thm:main} is divided into
three steps, each one corresponding to a proposition below.
\begin{proposition}[RSW-type result]\label{prop:weak RSW}
  Let $q\ge1$. If \eqref{eq:a} holds for $\phi_{p,q}^1$, then 
  \begin{equation}\label{eq:b}\limsup_{n\rightarrow \infty} \> \phi_{p,q}^1\left[
      \calC_h(3n,n) \right] > 0\tag{$\mathcal B$}.
  \end{equation}
\end{proposition}
The proof of this proposition is based on a Russo-Seymour-Welsh (RSW) type
argument used in \cite{Tas14b} in the context Voronoi percolation.
Interestingly, this part of the argument uses the FKG inequality only. The cost
is that we obtain that the limsup only is strictly positive, instead of the
infimum. Nevertheless, as we will see this will be sufficient for our purpose.
\begin{proposition}[Sharp threshold for crossing probabilities]\label{prop:high probability horizontal crossing}
Let $q\ge1$ and $p\in[0,1]$. For any $\delta\in(0,1-p]$, there exists
$c=c(p,q)>0$ such that  for every $n\ge 1$
\begin{equation}
\phi_{p+\delta,q}^1 [\calC_h(2n,n)] \geq 1 - \frac{1}{\phi_{p,q}^1[\calC_h(3n,n)]} n^{-c\delta}.\label{eq:3}
\end{equation}

\end{proposition}
This type of statement has been widely used in the recent studies of phase
transition. It is based on sharp threshold results going back to
\cite{BouKahKal92} (we provide more details before the proof). The novelty of
the proof of the proposition above lies in the fact that we do not need to
symmetrize the event to which we wish to apply the sharp threshold. More
precisely, in \cite{BefDum12}, a similar sharp threshold result is obtained by
first proving the result on the torus, and then bootstrapping it to the plane.
Here we present a new method based on the sharp-threshold theorem of
\cite{GraGri11} that allows us to circumvent this difficulty.
\begin{proposition}[Bootstraping to exponential decay]\label{prop:exponential decay2}
Let $q\ge1$ and $p\in[0,1]$. If 
\begin{equation}\label{eq:c}\limsup_{n\rightarrow \infty} \> \phi_{p,q}^1\left[ \calC_h(2n,n) \right] =1\tag{$\mathcal C$},\end{equation}
 then for any $\delta\in(0,1-p]$, there exists $c=c(\delta, p, q)>0$ such that for any $n\ge1$,
$$\phi_{p+\delta,q}^1[\calC_h(2n, n)] \ge 1 - e ^{-cn}.$$
\end{proposition}
These three propositions imply Theorem~\ref{thm:main} readily. Indeed, \eqref{eq:a} at $p$ and Proposition~\ref{prop:weak RSW} applied at $p$ imply \eqref{eq:b} at $p$. Proposition~\ref{prop:high probability horizontal crossing} applied to $p$ and $\delta/2$ implies \eqref{eq:c} at $p+\delta/2$. Proposition~\ref{prop:exponential decay2} applied to $p+\delta/2$ and $\delta/2$  concludes the proof.
\bigbreak
\paragraph{Notation}From now on, we fix $q\ge1$ and write $\phi_p$ instead of
$\phi_{p,q}^1$, and $\phi_{G,p}$ instead of $\phi_{G,p,q}^1$. Furthermore, for
an automorphism $T$  of the square lattice and an event $A$, we define the {\em image} of $A$ by $T$ as the set
$$B:=\{\omega:T^{-1}(\omega)\in A\},$$
where $T^{-1}(\omega)_e:=\omega_{T^{-1}(e)}$. Note that $\phi_p$ is symmetric under any automorphism $T$ of the lattice, and therefore $A$ and the image of $A$ by $T$ have the same probability. {\bf We will often refer to this fact without mentioning it}. Since we will use it extensively, we also introduce $\tau_xA$ to be the image of $A$ by the translation $\tau_x$ of vector $x$. 
\bigbreak
\paragraph{Organization}The paper is organized as follows. The next three sections correspond respectively to the proofs of the last three propositions. The last section is devoted to the proof of Theorem~\ref{main:corollary}. We included bibliographical comments and discussions on the scope of the proofs and the comparison with existing arguments at the end of each section. 
\section{Proof of Proposition~\ref{prop:weak RSW}}

Below, we will make extensive use of the FKG inequality (see
\cite[Theorem~3.8]{Gri06}) which states
that \begin{equation}\label{eq:FKG}\phi_p[A\cap B]\ge
  \phi_p[A]\phi_p[B]\tag{$\mathrm{FKG}$},\end{equation} for any
two increasing events $A$ and $B$. We recall that an event $A$ is increasing if for every
$\omega\in A$ and $\omega'\ge \omega$ (for the product ordering on
$\{0,1\}^{E_{\bbZ^2}}$), we also have  $\omega'\in A$. \bigbreak Since
we will use it repeatedly, let us recall a classical fact. Let $k,\ell>0$, and
$n,m\ge1$. If $m\ge n/k$, then
   \begin{align}\label{eq:combination}
\phi_p\left[ \calC_h(\ell n,n)\right] &\ge \phi_p\left[\calC_h(n+m,n)\right]^{2k\ell}.
\end{align}
In order to obtain this inequality, apply the FKG inequality to the events
$\tau_{(2mj,0)}\calC_h(n+m,n)$ for $2j\in\llb -k\ell,k\ell\rrb$ and
$\tau_{(2mj,0)}\widetilde \calC_h(n,n)$ with $2j\in\llb -k\ell+1,k\ell-1\rrb$,
where $\widetilde \calC_h(n,n)$ is the image of $\calC_h(n,n)$ by the rotation
of angle $\pi/2$ around the origin.
 \bigbreak For $n\ge 1$ and $-n \leq \alpha
\leq \beta \leq n$, define $\calH_n(\alpha,\beta)$ to be the event (illustrated
on Fig.~\ref{crossing:fig:eventH}) that there exists an open path in $\llb
-n,n\rrb^2$ from $\{-n\}\times\llb -n,n\rrb$ to
$\{n\}\times\llb\alpha,\beta\rrb$.
\begin{figure}[htbp]
\centering
\hfill
\begin{minipage}[t]{.44\linewidth}
\centering
  \includegraphics[width=.9\linewidth]{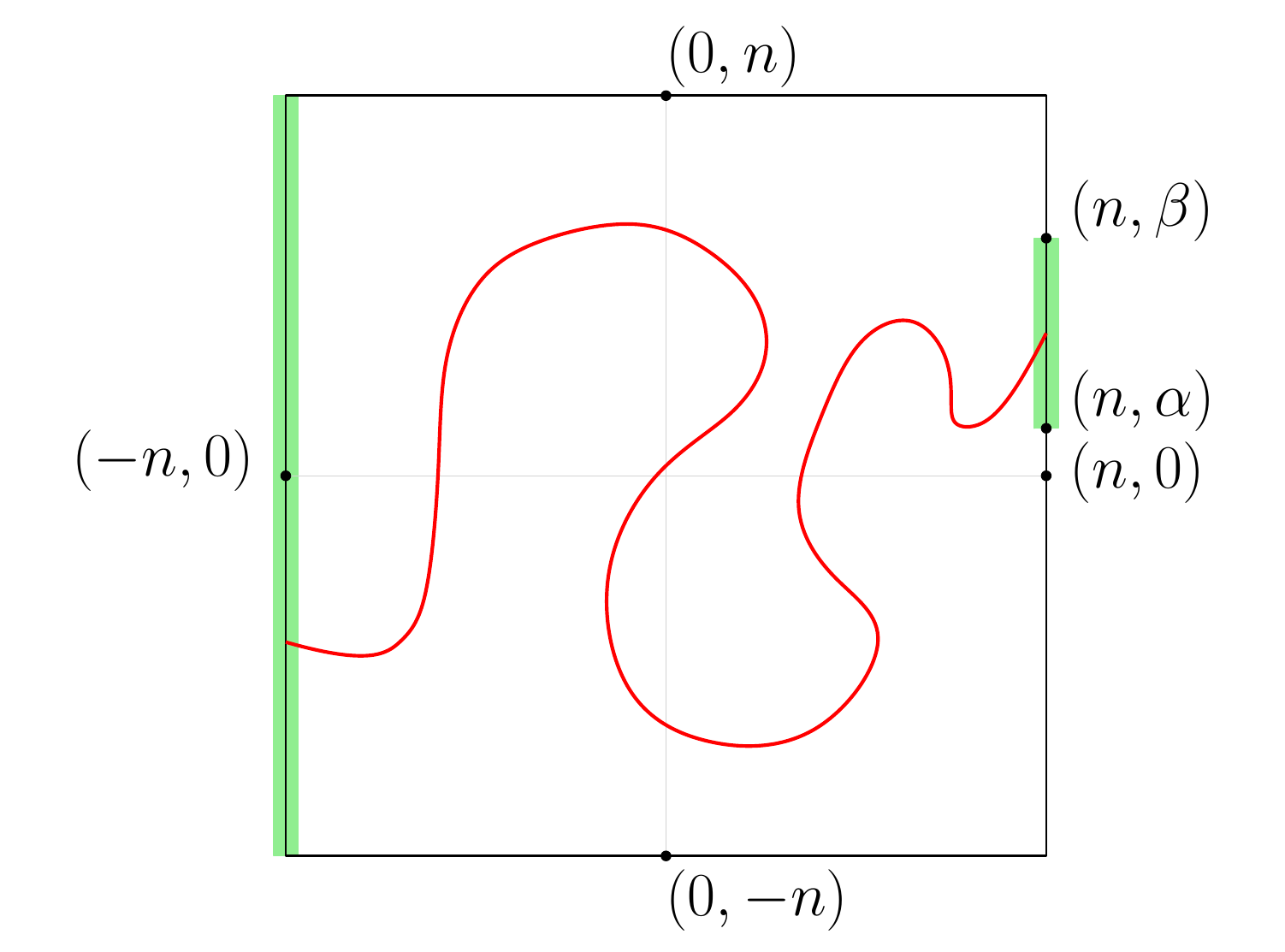}
  \caption{A diagrammatic representation of the event $\calH_n (\alpha,\beta)$}
  \label{crossing:fig:eventH}
\end{minipage}
\hfill
\hfill
\begin{minipage}[t]{.44\linewidth}
\centering
  \includegraphics[width=0.9\linewidth]{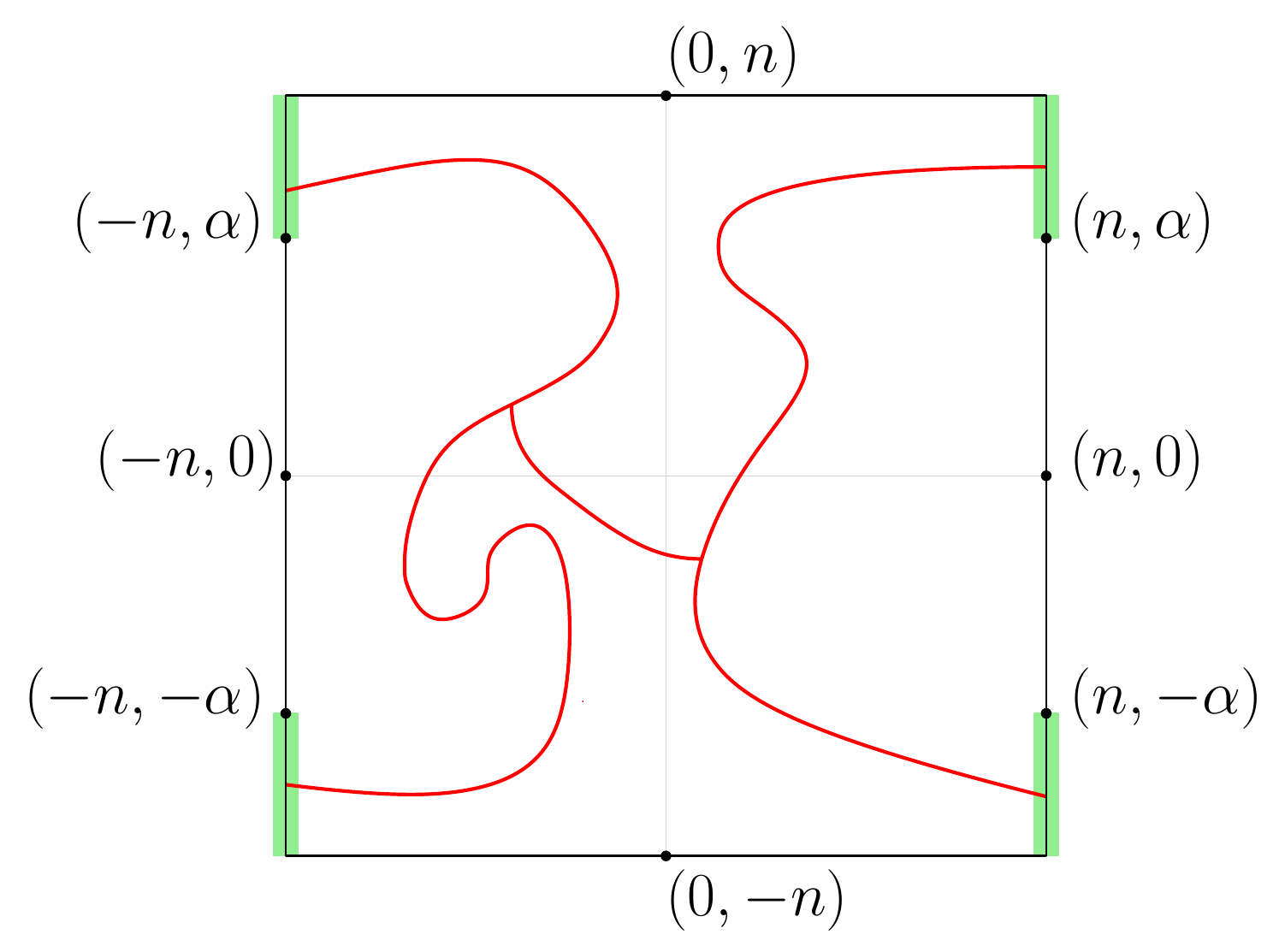}
  \caption{A diagrammatic representation of the event $\calX_n (\alpha)$}
  \label{crossing:fig:eventX}
\end{minipage}
\hfill
\end{figure}
The symmetry with respect to the $x$-axis implies that 
  $$\phi_p[ {\calH_n ( 0,n)}]\ge \tfrac12\phi_p[\calC_h(n,n)].$$
  Therefore, \eqref{eq:a} implies that 
\begin{equation}\label{eq:a'}c_0=\inf_{n\in \mathbb{N}}\phi_p[ {\calH_n ( 0,n)}] >0\tag{$\mathcal A'$}.\end{equation} 
 

Fix $p$ such that \eqref{eq:a} is satisfied. Let $h_n(\alpha):=\phi_p[\calH_n(\alpha,n)]$. Note that $h_n$ is strictly
decreasing and $h_n(0) \ge c_0 $. Define $\alpha_n:= \max \lbrace \alpha \le n:
\: h_n(\alpha)\ge c_0/2 \rbrace$ so that 
    \begin{align}\label{eq:alpha2}
       & \phi_p[{\calH}_n(\alpha,n)] \ge
    c_0/2 &&\text{for every $0\le\alpha\le\alpha_n$ and}\\
   & \phi_p[{\calH}_n(0,\alpha)] \ge
    c_0/2  &&\text{for every $\alpha_n\le \alpha\le n$}.\label{eq:alpha1}
    \end{align}
    Eq.~\eqref{eq:alpha2} follows directly from the definition of $\alpha$. To show Eq.~\eqref{eq:alpha1}, first notice that one can assume
    $\alpha_n\le\alpha<n$, and then observe that in this case
    \begin{align}
      \label{eq:1}
      \phi_p[{\calH}_n(0,\alpha)]\ge\underbrace{\phi_p[\calH_n ( 0,n)]}_{\ge c_0}-
      \underbrace{\phi_p[\calH_n(\alpha+1,n)]}_{=h_n(\alpha+1)\le c_0/2}\ge c_0/2.
    \end{align}
    For $0\le \alpha \leq n$, let $\calX_{n}(\alpha)$ be the event (illustrated
    on Fig.~\ref{crossing:fig:eventX}) that there exists a connected component
    of $\omega$ in $\llb -n,n\rrb^2$ that intersects the line segments
    $\{-n\}\times \llb-n,-\alpha\rrb$, $\{-n\}\times \llb\alpha,n\rrb$,
    $\{n\}\times \llb-n,-\alpha\rrb$, and $\{n\}\times \llb\alpha,n\rrb$.
  
   Let $\calH^1$, $\calH^2$, $\calH^3$ and $\calH^4$ be the events obtained from $\calH_n(\alpha,n)$ by taking the successive images by the orthogonal symmetries with respect to the $x$ and $y$-axis. 
  Let $\calF$ be the image of $\calC_h(n,n)$ by the rotation of angle $\pi/2$ around the origin. 
  Then for $\alpha \leq \alpha_n$, we find
\begin{align*} \label{eq.Xnbounded}
  \phi_p[\calX_n(\alpha)] &\stackrel{\hphantom{\,\rm (FKG)\,}}\geq  \phi_p[\calH^1\cap\calH^2\cap\calH^3\cap\calH^4\cap\calF]\\
  &\stackrel{\,\rm (FKG)\,}\ge \phi_p[\calH_n(\alpha,n)]^4\cdot \phi_p[\calC_h(n,n)]\\
  &\stackrel{\eqref{eq:alpha2}\text{,}\eqref{eq:a}}\ge \left(\frac{c_0}{4}\right)^4\cdot c_0.
\end{align*}
As a consequence, 
   \begin{equation}\label{eq:a''}c_1:=\inf_{n\in \mathbb{N}}\phi_p[\calX_n(\alpha_n)] >0\tag{$\mathcal A''$}.\end{equation} 
We now divide the proof in three cases:
\paragraph{Case 1: $\alpha_{n+\lfloor n/2\rfloor}\ge 2\,\alpha_{n}$ for $n$ large enough.} In such case, $\alpha_n$ would not be bounded by $n$ for every $n$, which is absurd. Thus, this case does not occur.

\paragraph{Case 2: $\alpha_n\ge n/2$ for infinitely many $n$.}  In such case, pick $n$ so that $\alpha_n\ge n/2$. Set $\calX$ to be the image of $\calX_n(\alpha_n)$ by rotation of angle $\pi/2$ around the origin. 
Thus
  \begin{align*}
     \phi_p\left[ \calC_h(3n,n)\right] &\stackrel{\hphantom{\rm (FKG)}}\ge \phi_p\left[\bigcap_{i=-3}^3\tau_{(in,0)}\calX
     \right]\\
     & \stackrel{\rm (FKG)}\ge\phi_p[\calX_n(\alpha_n)]^7\ \stackrel{\eqref{eq:a''}}\ge\  c_1^7. 
\end{align*}
As a consequence, the existence of infinitely such $n$ implies \eqref{eq:b}.
%

 \medbreak

\paragraph{Case 3: $\alpha_n< \lfloor n/2\rfloor $ and $\alpha_{n+\lfloor n/2\rfloor}< 2\,\alpha_{n}$ for infinitely many $n$.}  Fix $n$ satisfying the two previous inequalities. To lighten the notation, set $m:=\lfloor n/2\rfloor$ and $N:=n+m$. Consider the
  two square boxes
  \begin{align*}R&:=(-m,-\alpha_n)+\llb -N,N\rrb^2,\\
  R'&:=(m,-\alpha_n)+\llb -N,N\rrb^2.\end{align*}
  Let $\calE=\tau_{(-m,-\alpha_n)}\calH_N(0,2\alpha_n)$ and $\calE'$ be the image of $\calE $ by the reflection with respect to the $y$-axis. Then, $\llb -n,n\rrb^2$ is included in $R\cap R'$ since $\alpha_n\le m$ and it follows that $\calX_{n}(\alpha_n) \cap \calE\cap\calE'\subset {\calC}_h(N+2m,N)$. This, together with $\alpha_N < 2\alpha_n$, implies
   \begin{align}
    \phi_{p}[\calC_h(N+2m,N)]&\stackrel{\ \hphantom{\rm (FKG)}\ }\geq \phi_p[\calX_{n}(\alpha_n) \cap \cal E\cap\cal E']\nonumber \\
 &\stackrel{\ \rm (FKG)  \ }\ge   \phi_p[\calX_n(\alpha_n)]\cdot \phi_p[{\calH}_N(0,2 \alpha_n)]^2\nonumber \\
 &\stackrel{\eqref{eq:a''},\eqref{eq:alpha1}}\ge c_1\cdot (\tfrac12c_0)^2.\label{eq:beta}
  \end{align}
We can apply \eqref{eq:combination} to $N$ and $2m$ to deduce that 
 $\phi_p[\calC_h(3N,N)]\ge  c_2$
for some $c_2>0$ which depends on $c_1$ and $c_0$ only. Therefore, the existence of infinity many such $N$ implies \eqref{eq:b}. 
  \begin{figure}[htp]
    \centering
    \includegraphics[width=.58\linewidth]{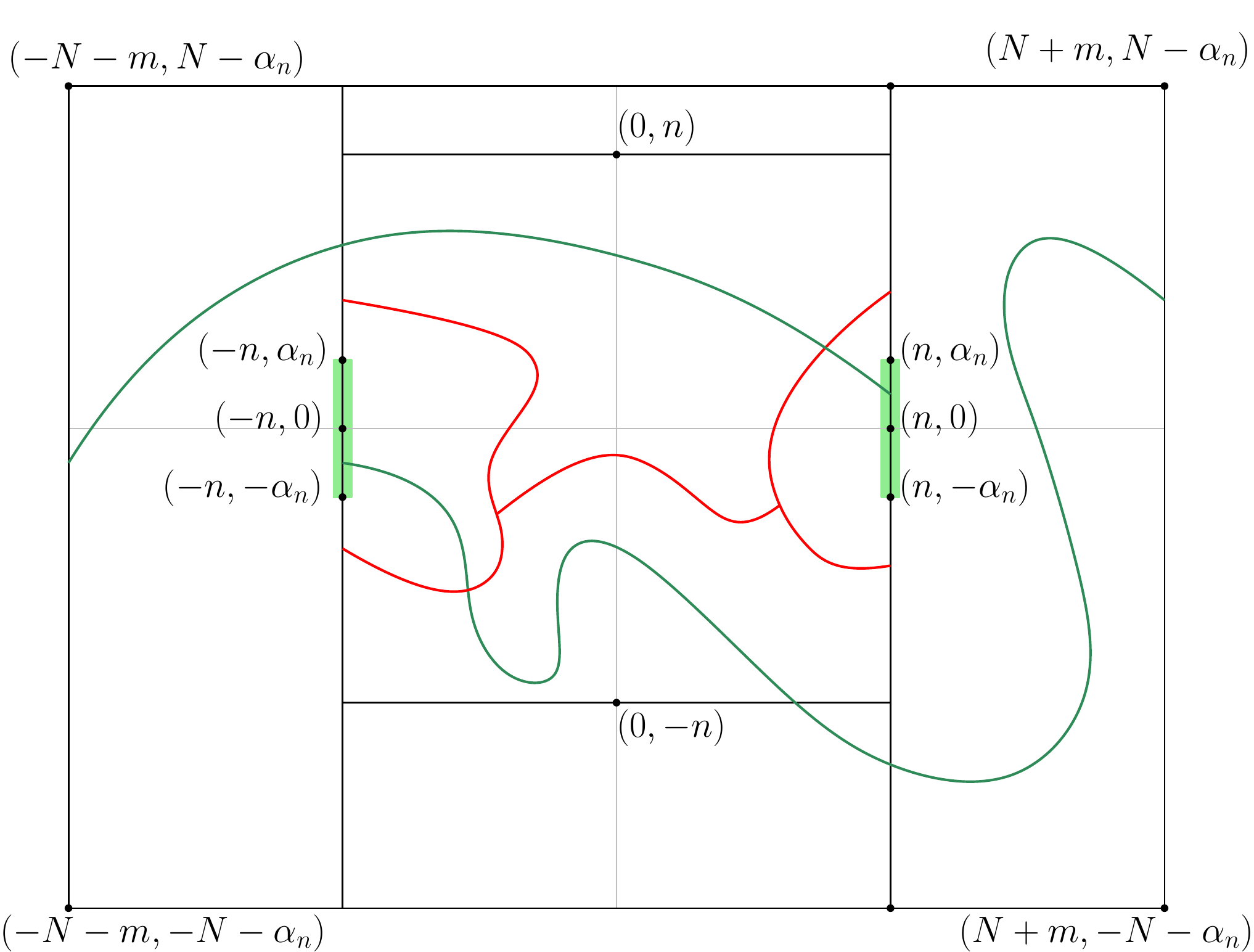}
    \caption{The simultaneous occurrence of $\calX_n(\alpha_n)$, $\calE$ and $\calE'$
      implies the existence of a horizontal crossing in $\llb -(N+2m),(N+2m)\rrb\times\llb -N,N\rrb$.}
    \label{crossing:fig:constr3}
  \end{figure}

\bigbreak

\noindent\textbf{Remarks and comments.}

\noindent{\bf 1.} In terms of the percolation model, as mentioned in \cite{Tas14b}, the proof only requires the FKG inequality. In particular, it does not involve independence, duality, or a conditioning with respect to the highest crossing.
\medbreak\noindent
{\bf 2.} In terms of the graph, planarity is clearly essential. The proof also requires the graph to have several symmetries: namely the axial symmetry with respect to the axis, the invariance under translation and the symmetry under the rotation of angle $\pi/2$ around the origin.
\medbreak\noindent
{\bf 3.} The fact that we prove a result for the limsup only is a draw back, but this will not be relevant for the rest of the proof.
\medbreak\noindent
{\bf 4.} One may prove stronger bounds on crossing probabilities, see e.g.\@ \cite{DumSidTas14,Dum13}.

\section{Proof of Proposition~\ref{prop:high probability horizontal crossing}}
Our goal is to prove a sharp threshold for the probability of an open path from
left to right. The starting point of the proof of such a statement is usually
the following simple differential equality \cite[Theorem~3.12]{Gri06}. Let $G$
be a finite subgraph of $\bbZ^2$ and $A$ an increasing event depending on the
states of edges in $G$ only. We have for every $p\in (0,1)$,
\begin{equation}
  \label{eq:Russo}
  \frac{d}{dp}\phi_{G,p}[A]=\frac1{p(1-p)}\sum_{e\in E_G}J_{A,G,p}(e),
\end{equation}
where $J_{A,G,p}(e):=\phi_{G,p}[{\bf
  1}_A\omega_e]-\phi_{G,p}[A]\phi_{G,p}[\omega_e]$.
 
In order to prove a sharp threshold result, we will use the following result,
which is a straightforward consequence of \cite[Theorem 5.1]{GraGri11} (the
original result concerns a more general class of measures than that of the
random-cluster model). There exists a constant $c=c(p,q) > 0$ such that uniformly
in $G$ and $A$,
\begin{equation}\label{eq:ggsh0}
  \sum_{e\in E_G}J_{A,G,p}(e)\geq  c	\,\phi_{G,p}[A]( 1 - \phi_{G,p}[A])
  \log \left(\frac{c}{ \max\{ J_{A,G,p}(e):e\in E_G\}}\right).
  \end{equation}
  In order to avoid confusion, let us mention that \eqref{eq:ggsh0} is usually
  stated in terms of the notion of influence of an edge $e$ which is, up to
  constant, related to $J_{A,p}(e)$.
  At this point, 
  \eqref{eq:Russo} together with \eqref{eq:ggsh0} imply
  \begin{equation}\label{eq:ggsh1}
	  	\frac{d}{dp}\phi_{G,p}[A]\geq  c\cdot 	\,\phi_{G,p}[A]( 1 - \phi_{G,p}[A])
		\cdot f_{G,p}(A),
  \end{equation}
  where 
  $$f_{G,p}(A):= \max \left\lbrace 
    \log \left(\frac{c}{ \max\{J_{A,G,p}(e):e\in E_G\}} \right) , \sum_{e\in
      E_G} J_{A,G,p}(e) \right\rbrace.$$ We plan to apply this inequality to our
  context by proving that for our event, the influence of any edge $e$ is small.
  We face a tiny technical difficulty: we are working directly in infinite
  volume with $\phi_p$. For this reason, we introduce the following lemma which
  states an integrate infinite-volume version of \eqref{eq:ggsh1} (its proof
  can be skipped in a first reading). Let $J_{A,p}(e):=
  \phi_p[{\bf 1}_A\omega_e]-\phi_p[A]\phi_p[\omega_e]$ and
  $$f_{p}(A):= \max \left\lbrace 
	\log \left(\frac{c}{ \sup\{J_{A,p}(e):e\in E_{\bbZ^2}\}} \right) ,  \sum_{e\in E_{\bbZ^2}} J_{A,p}(e)
	\right\rbrace\in[0,\infty].$$
\begin{lemma}\label{thm:ivggsh1}
  For every $p\in (0,1)$, $\delta\in[0,1-p)$, and every event $A$
  depending on finitely many edges,
	\begin{equation}\label{eq:ivggsh1}
       \log \left( \frac{\phi_{p+\delta}[A]}{1- \phi_{p+\delta}[A]} \cdot \frac{1-\phi_{p}[A]}{\phi_{p}[A]} \right)     
	  	\geq c \int_{p}^{p+\delta} f_s(A) \: ds .
  \end{equation}
\end{lemma}

\begin{proof} 
  For a positive integer $n$, let $\Lambda_n := \llbracket -n,n \rrbracket ^2$.
  Choosing $G=\Lambda_n$ and integrating \eqref{eq:ggsh1} between $p$ and
  $p+\delta$ gives
\begin{equation*}
\log \left( \frac{\phi_{\Lambda_n,p+\delta}[A]}{1- \phi_{\Lambda_n,p+\delta}[A]} \cdot \frac{1-\phi_{\Lambda_n,p}[A]}{\phi_{\Lambda_n,p}[A]} \right)     
\geq c \int_{p}^{p+\delta} f_{\Lambda_n,s}(A)  \: ds   .
\end{equation*}
The definition of the infinite-volume measure implies that the left-hand side converges to the left-hand side of \eqref{eq:ivggsh1}. Thus, Fatou's lemma implies
\begin{equation*} 
\log \left( \frac{\phi_{p+\delta}[A]}{1- \phi_{p+\delta}[A]} \cdot \frac{1-\phi_{p}[A]}{\phi_{p}[A]} \right)     
\geq c \int_{p}^{p+\delta} \liminf_{n\rightarrow\infty}f_{\Lambda_n,s}(A) ds ,
\end{equation*}
so that it suffices to show that for any increasing event $A$ depending on finitely many edges, 
\begin{equation}\label{eq:9}\liminf_{n\rightarrow\infty} f_{\Lambda_n,s}(A)\ge f_s(A).\end{equation}
First, fix $k\le n$ and observe that
$$f_{\Lambda_n,s}(A)\ge  \sum_{e\in E_{\Lambda_n}}J_{A,\Lambda_n,s}(e)\ge
 \sum_{e\in E_{\Lambda_k}}J_{A,\Lambda_n,s}(e)$$
since $J_{A,\Lambda_n,s}(e)\ge 0$ by the FKG inequality. The definition of the infinite-volume measure implies that $J_{A,\Lambda_n,s}(e)$ tends to $J_{A,s}(e)$ as $n$ tends to infinity. Letting $n$ and then $k$ tend to infinity implies
\begin{equation}\label{eq:10}
  \liminf_{n\rightarrow\infty} f_{\Lambda_n,s}(A)\ge \sup_{k\ge1} \sum_{e\in
    E_{\Lambda_k}}J_{A,s}(e)=\sum_{e\in E_{\bbZ^2}}J_{A,s}(e).
\end{equation}
At this point, \eqref{eq:9} (and therefore the claim) would follow from
\begin{align}\label{eq:11}\liminf_{n\rightarrow\infty} f_{\Lambda_n,s}(A)&\ge\log\left(\frac{c}{\sup\{J_{A,s}(e):e\in  E_{\bbZ^2}\}}\right),\end{align} 
which we prove now. Let $k\le n$ and $e\notin E_{\Lambda_k}$, the domain Markov property and the comparison between boundary conditions (see \cite{Gri06}) imply that 
\begin{align*}J_{A,\Lambda_n,s}(e)&= \big(\phi_{\Lambda_n,s}[{\bf 1}_A|\omega_e=1]-\phi_{\Lambda_n,s}[A]\big)\phi_{\Lambda_n,s}[\omega_e]\\
&\le \phi_{\Lambda_k,s}[A]-\phi_{s}[A].\end{align*}
Now, since $ \sup\{J_{A,s}(e):e\in E_{\bbZ^2}\}>0$, we can use the definition of the infinite-volume measure to choose $k$ large enough so that 
$$\phi_{\Lambda_k,s}[A]-\phi_{s}[A]\le \sup\{J_{A,s}(e):e\in E_{\bbZ^2}\}.$$
We deduce that
$$\sup_{n\ge k}\ \sup\{J_{A,\Lambda_n,p}(e):e\in E_{\Lambda_n}\setminus E_{\Lambda_k}\} \le \sup\{J_{A,s}(e):e\in  E_{\bbZ^2}\}.$$
But the definition of the infinite-volume measure also implies that
$$\limsup_{n\rightarrow\infty} \: \sup\{J_{A,\Lambda_n,p}(e):e\in E_{\Lambda_k}\} \le \sup\{J_{A,s}(e):e\in  E_{\bbZ^2}\}.$$
Combining the last  two displayed inequalities gives 
\begin{align*}
  \liminf_{n\rightarrow\infty} f_{\Lambda_n,s}(A)&\ge \log \left(\frac{c}{\displaystyle\limsup_{n\rightarrow \infty}\ \sup\{J_{A,\Lambda_n,p}(e):e\in E_{\Lambda_n}\}} \right)\\
  &\ge\log\left(\frac{c}{\sup\{J_{A,s}(e):e\in E_{\bbZ^2}\}}\right)
\end{align*}
which is \eqref{eq:11}. \end{proof}

The previous lemma will be combined with the following lemma. Let $A_k=A_k(n):=\calC_h(k,n)$.
\begin{lemma}\label{lem:averaging}
For any $n$ large enough,
$$\sum_{k=2n}^{3n}f_p\big(A_k\big)\ge \tfrac16n\log n.$$
\end{lemma}
Before proving this lemma, let us show how it can be used to conclude the proof.
Applying the above inequality to the right-hand-side of \eqref{eq:ivggsh1}, we
obtain
$$\sum_{k=2n}^{3n}\log \left( \frac{\phi_{p+\delta}[A_k]}{1- \phi_{p+\delta}[A_k]} \cdot \frac{1-\phi_{p}[A_k]}{\phi_{p}^1[A_k]} \right) \ge \frac{c\delta}{6} n\log n.$$
As a consequence, there exists $k \in\llb2n, 3n\rrb$ such that 
$$ \log \left( \frac{\phi_{p+\delta}[A_k]}{1- \phi_{p+\delta}[A_k]} \cdot \frac{1-\phi_{p}[A_k]}{\phi_{p}[A_k]} \right) \ge \frac{c\delta}{6} \log n,$$
or after applying the exponential,
\begin{align}\label{eq:aaaa}
\frac{1}{1- \phi_{p+\delta}[A_k]\phi_{p}[A_k]} ~\ge~ \frac{\phi_{p+\delta}[A_k]}{1- \phi_{p+\delta}[A_k]} \cdot \frac{1-\phi_{p}[A_k]}{\phi_{p}[A_k]} ~ \ge~ n^{c\delta/6}.
\end{align}
The fact that $A_{3n}\subset A_k\subset A_{2n}$ implies that
\begin{align*}
\phi_{p+\delta}[A_{2n}]\ge \phi_{p+\delta}[A_{k}] &\stackrel{\eqref{eq:aaaa}}\ge 1 - \frac{1}{\phi_{p}[A_{k}]} n^{-c\delta/6}\ge 1 - \frac{1}{ \phi_{p}[A_{3n}]} n^{-c\delta/6},
\end{align*}
which is the claim of Proposition~\ref{prop:high probability horizontal crossing}. We therefore only need to prove Lemma~\ref{lem:averaging}.
\begin{proof}[Proof of Lemma~\ref{lem:averaging}]
Let $e\in E_{\bbZ^2}$. We have\begin{align*}
J_{A_k, p}(e)&=\phi_{p}[{\bf 1}_{A_k}\omega_e]-\phi_{p}[A_k]\phi_{p}[\omega_e]\\
&= \phi_{p}[{\bf 1}_{\tau_{(1,0)} A_k}\omega_{\tau_{(1,0)}  e}]-\phi_{p}[A_k]\phi_{p}[\omega_{\tau_{(1,0)}  e}]\\
&= J_{A_k, p}(\tau_{(1,0)}  e)+ \phi_{p}[({\bf 1}_{\tau_{(1,0)}  A_k}-{\bf 1}_{A_k})\omega_{\tau_{(1,0)}  e}].
\end{align*}
Since both ${\bf 1}_{\tau_{(1,0)}  A_k}$ and ${\bf 1}_{A_k}$ are equal to 1 (respectively 0) on $A_{k+1}$ (respectively the complement of $\tau_{(1,0)}A_{k-1}$), we deduce that
\begin{equation}| J_{A_k,p}(e)- J_{A_k,p}(\tau_{(1,0)}  e)|\le \phi_{p}[A_{k-1}]- \phi_{p}[A_{k+1}].\label{eq:aaab}\end{equation}
Since the sequence of events $(A_k)$ is decreasing, we deduce that 
$$\sum_{k=2n}^{3n}\phi_p[A_{k-1}]-\phi_p[A_{k+1}]=\phi_{p}[A_{2n-1}]+\phi_{p}[A_{2n}]-\phi_{p}[A_{3n}]
-\phi_p[A_{3n+1}]\le 2,$$ and therefore there exists a set $\calK\subset
\llbracket 2n,3n \rrbracket$ of cardinality at least $n/2$ such that for any
$k\in \calK$ and any $e\in E_{\bbZ^2}$,
$$| J_{A_k,p}(e)- J_{A_k,p}(\tau_{(1,0)}  e)|\stackrel{\eqref{eq:aaab}}\le \phi_{p}[A_{k-1}]-\phi_{p}[A_{k+1}]\le \frac4n.$$
Now, for $k\in \calK$, one has that 
\begin{equation}\label{eq:aaac}f_p(A_k)\ge \frac13\log n\end{equation} as can be seen by dividing into the two following cases:
\begin{description}
\item[{\bf Case 1:}] 
For every  $e\in E_{\bbZ^2}$, $ J_{A_k, p}(e)\le \frac c{n^{1/3}}$. Then the definition of $f$ gives
$$f_{p}(A_k) \ge \tfrac{1}{3} \log n.$$
\item[{\bf Case 2:}] 
There exists $e\in E_{\bbZ^2}$ such that $J_{A_k, p}(e)\ge \frac c{n^{1/3}}$. In
such case, $ J_{A_k, p}(\tau_{(s,0)}e)\ge \frac c{2n^{1/3}}$ for any edge $\tau_{(s,0)} e$ with $s\le \frac{c}{8} n^{2/3}$. Therefore,
\begin{align*}
f_{p}(A_k) &\ge \sum_{s=0}^{\lfloor \tfrac{c}{8}n^{2/3}\rfloor}
J_{A_k,p}(\tau_{(s,0)}e)\ge \sum_{s=0}^{\lfloor \tfrac{c}{8}n^{2/3}\rfloor}
\tfrac c2 n^{-1/3}\ge \tfrac{c^2}{16}n^{1/3} \ge \tfrac13\log n.
\end{align*}
provided $n$ is large enough.
\end{description}
Summing over every $k\in\llb 2n,3n\rrb$ gives
$$\sum_{k=2n}^{3n} f_{p}(A_k) \ge \sum_{k \in \calK} f_{p}(A_k) \stackrel{\eqref{eq:aaac}}\ge \tfrac{1} {6} n \log n.$$
\end{proof}

\bigbreak

\noindent\textbf{Remarks and comments.}

\medbreak\noindent {\bf 1.} Sharp threshold theorems first emerged in the context of Boolean functions (see e.g. \cite{BouKahKal92} and references therein). 
A rudimentary version of the method of sharp threshold appears in the work of Russo and it was used it in the context of percolation  \cite{Rus81}.
Inequality~\ref{eq:ggsh0} was first used for percolation by Bollob\`as and Riordan \cite{BolRio06,BolRio06c}. It has since then found many other applications, thanks in particular to the generalization \cite{GraGri06} to the non-Bernoulli case. It was in particular instrumental in \cite{BefDum12,duminil2014phase}.
\medbreak\noindent
{\bf 2.} Lemma~\ref{lem:averaging} should hold without averaging, hence showing that $\calC_h(k,n)$ always satisfies a sharp threshold. Unfortunately, in order to prove such a result, one should prove that $\phi_p[\calC_h(k,n)]$ and $\phi_p[\calC_h(k+1,n)]$ are always polynomially close to each others. Note that if such a result would be true, we could apply it to $k=3n$ and obtain Proposition~\ref{prop:high probability horizontal crossing} with $3n$ instead of $2n$ on the left-hand side. 
\medbreak\noindent
{\bf 3.} Historically, sharp threshold theorems were often proved by using events which are invariant under translations. Such strategies required to work on a torus. In the case of models with dependencies, translating results obtained with periodic boundary conditions to results in the plane were often very difficult. Lemma~\ref{lem:averaging} enables to avoid this difficulty.
\medbreak\noindent
{\bf 4.} We expect Lemma~\ref{lem:averaging} to have further applications to the theory of sharp thresholds. Indeed, we used very little on the events $A_n$, namely that they were forming a decreasing sequence of events and that $A_n\Delta\tau_{(1,0)}(A_n)\subset A_{n-1}\setminus A_{n+1}$. Other sequences of events satisfy similar properties (for instance $B_n=\{\llb -n,n\rrb^2\leftrightarrow \infty\}$ or even $C_n=\{E_n\leftrightarrow F\}$ where $E_n$ is an increasing sequence of sets with $F\cap E_n=\emptyset$). 
\medbreak\noindent
{\bf 5.} The idea of proving that all $J_{A,p}(e)$ are small was already present in \cite{GraGri06}. Following their strategy, $J_{A_n,p}(e)$ would be bounded by the probability of having an open path going to distance $2n$. This bound could be used to prove that $p_c\le \sqrt q/(1+\sqrt q)$, but would be insufficient to prove Theorem~\ref{main:corollary}. Furthermore, the proof of Lemma~\ref{lem:averaging} is sufficiently elementary to represent a good alternative to the strategy proposed in \cite{GraGri06}.

\section{Proof of Proposition~\ref{prop:exponential decay2}}

In this section, $A^c$ denotes the complement of the event $A$. For an
increasing event $A$, let $H_{A^c}(\omega)$ be the Hamming distance from
$\omega$ to $A^c$ in $\{0,1\}^{E_G}$, defined as the minimal number of edges
that need to be turned to closed in order to go from $\omega$ to a configuration
in $A^c$. We use the following inequality, stated as a lemma.
\begin{lemma} \label{lem:hamming1}
Let $p\in[0,1]$, $\delta\in[0,1-p]$, and $A$ an increasing event depending on finitely many edges. Then
\begin{equation}\label{eq:hamming1}\phi_{p+\delta}[A]\ge 1- \exp\big(-4\delta\, \phi_{p}[H_{A^c}]\big).\end{equation}
\end{lemma}

\begin{proof}
  Let $G$ be a finite subgraph of $\bbZ^2$ such that $A$ is measurable with
  respect to the state of the edges of $G$. Let $|\omega|=\sum_{e\in
    E_G}\omega_e$. The facts that ${\bf 1}_{A^c}H_{A^c}=0$ and
  $|\omega|-H_{A^c}$ is increasing imply that
  \begin{align}\label{eq:hamming}
    \frac{d}{dp}\log(1-\phi_{G,p}[A])&\stackrel{ \ \eqref{eq:Russo} \ }=-\frac{\phi_{G,p}[{\bf 1}_{A}|\omega|]-\phi_{G,p}[A]\phi_{G,p}[|\omega|]}{p(1-p)(1-\phi_{G,p}[A])},\nonumber\\
    &\stackrel{\hphantom{\rm (FKG)}}=\frac{\phi_{G,p}[{\bf 1}_{A^c}|\omega|]-\phi_{G,p}[A^c]\phi_{G,p}[|\omega|]}{p(1-p)\phi_{G,p}[A^c]}\nonumber \\
    &\stackrel{\hphantom{\rm (FKG)}}=\frac{\phi_{G,p}[{\bf 1}_{A^c}(|\omega|-H_{A^c})]-\phi_{G,p}[A^c]\phi_{G,p}[|\omega|]}{p(1-p)\phi_{G,p}[A^c]}\nonumber\\
   &\stackrel{\rm (FKG)}\le \frac{\phi_{G,p}[A^c]\phi_{G,p}[|\omega|-H_{A^c}]-\phi_{G,p}[A^c]\phi_{G,p}[|\omega|]}{p(1-p)\phi_{G,p}[A^c]}\nonumber\\
   &\stackrel{\hphantom{\rm (FKG)}}\le -4\phi_{G,p}[H_{A^c}].
  \end{align}
Integrating \eqref{eq:hamming} between $p$ and $p+\delta$ and then taking the exponential, one obtain
$$\frac{1-\phi_{G,p+\delta}[A]}{1-\phi_{G,p}[A]}\le
\exp\Big(-4\int_p^{p+\delta}\phi_{G,s,q}^1[H_{A^c}] ds\Big)\le \exp\big(-4\delta\, \phi_{G,p}[H_{A^c}]\big).$$
In the second inequality we used that $\phi_{G,s}[H_{A^c}]$ is increasing in $s$.
We conclude the proof by using $\phi_{G,p}[A]\ge0$ and by letting $G$ tend to infinity.
\end{proof}
In the of Proposition~\ref{prop:exponential decay2}, a key ingredient will be
the following lemma, which can be seen as a generalization of
\eqref{eq:combination}. For a rectangle $R=\llb -a,a\rrb\times\llb -b,b\rrb$,
let
\begin{equation*}
\sfN(a,b):=\max\{n\ge1:\exists n\text{ disjoint paths in $R$ from }\{-a\}\times\llb -b,b\rrb\text{ to }\{a\}\times\llb -b,b\rrb\}.
\end{equation*}
Note that $\sfN(a,b)$ is equal to $H_{A^c}$ where $A$ is the event that there
exists a path from left to right in $R$ (this follows from Menger's
mincut-maxflow theorem).

As in \eqref{eq:combination}, we have 
\begin{equation*}
  \label{eq:6}
  \phi_p[\sfN(\ell n,n)\ge1]\ge  \phi_p[\sfN(2n,n)\ge1]^{2(\ell-2)+1}
\end{equation*}
for every $p\in[0,1]$, $n\ge1$, $\ell\ge 2$. We prove below that this argument
also works for several disjoint paths, and we can replace the ``$1$''s in the
two events estimated above by an arbitrary number.
\begin{lemma}
  \label{sec:lemGluingManyPaths}
  Let $0\le p\le 1$, $n,u\ge 1$, $\ell\ge 2$. We have
\begin{equation*}
  \label{eq:7}
  \phi_p[\sfN(\ell n,n)\ge u]\ge  \phi_p[\sfN(2n,n)\ge u]^{2(\ell-2)+1}.
\end{equation*}
\end{lemma}
\begin{proof}
  We only prove it for $\ell=3$. The more general statement above follows by
  induction. Consider the rectangles $R_j=\llbracket(-3+j)
  n,(1+j)n\rrbracket\times \llbracket-n,n \rrbracket$, $j=0,2$, and $R_1=\llbracket-n,n \rrbracket \times \llbracket-n,3n\rrbracket$. Let $E$ be
  the event that both $R_0$ and $R_2$ are crossed horizontally by $u$ disjoint
  paths, and $R_1$ is crossed vertically by $u$ disjoint paths. By the FKG
  inequality, we have
  \begin{equation*}
    \label{eq:12}
    \phi_p[E] \ge \phi_p[\sfN(2n,n)\ge u]^3. 
  \end{equation*}
  Now, observe that on the event $E$, there must exist at least $u$ disjoint
  paths crossing horizontally the rectangle $R=\llbracket -3 n,3 n \rrbracket\times \llbracket -n,n \rrbracket$. Indeed,
  if we close less than $u$ edges in $R$ then $R_0$ and $R_2$ remains crossed
  horizontally and $R_1$ remains crossed vertically, and therefore $R$ is also
  crossed horizontally. Therefore, by Menger's Mincut-Maxflow theorem, there
  must exist at least $u$ disjoint paths crossing horizontally the rectangle
  $R$. Thus we obtain
  \begin{equation*}
    \label{eq:13}
    \phi_p[\sfN(3n,n)\ge u]\ge  \phi_p[E]\ge \phi_p[p[\sfN(2n,n)\ge u]^3.
  \end{equation*}
\end{proof}

We are ready to proceed with the proof of proposition~\ref{prop:exponential
  decay2}. We assume that 
\begin{equation}
  \label{eq:14}
  \limsup_{n\rightarrow \infty} \> \phi_{p}\left[ \calC_h(2n,n) \right] =1,
\end{equation}
holds, and fix $\delta>0$. 

We choose $i_0\ge1$ large enough such that for every $i\ge i_0$,
\begin{equation}
  \label{eq:C1}
  1-\exp(-\tfrac{\delta2^{i-9}}{(i+1)^4})\ge \frac12.
\end{equation}
Then, by \eqref{eq:14}, we can pick $n_0\ge 1$ such that
\begin{equation}
  \label{eq:16}
  \phi_p[\calC_h(2n_0,n_0)]^{4\cdot 2^{2i_0}}\ge\frac12.
\end{equation}

The proof of Proposition~\ref{prop:exponential decay2} is based on the following Lemma.
\begin{lemma}\label{lem:1a} For $i\ge0$ set $p_i=p+\delta\sum_{1\le j\le i}
  \tfrac1{i^2}$, $c_i=1-\tfrac14\sum_{1\le j\le i}\tfrac1{i^2}$ and $K_i=2^{i} n_0$. For every
  $i\ge i_0$, we have 
\begin{equation}
  \label{eq:8}
  \phi_{p_i}[\sfN(2K_{i},K_i)\ge c_i2^i]\ge 1/2.
\end{equation}
\end{lemma}
Before proving this lemma, let us explain how we finish the proof. First, observe that $p_i\le p+2\delta$ and $c_i\ge 1/2$ for any
$i\ge0$. 
Therefore, by monotonicity, \eqref{eq:8} 
implies for every $i\ge i_0$,
$$\phi_{p+2\delta}[\sfN(2K_{i},K_i)]\ge  2^{i-2}.$$ 
Since  $\sfN(2K_i,K_i)(\omega)$ is equal to  $H_{\calC_h^c(2K_{i},K_i)}$,  
Lemma~\ref{lem:hamming1} implies that for $i \ge i_0$,
\begin{equation*}
  \phi_{p+3\delta}[\calC_h \left( 2K_{i}, K_i \right) ]\ge 1-e^{-\delta  2^i}.\end{equation*}
Let $n\ge 2^{i_0}n_0$ and choose $i$ such that $2^{i} n_0 \leq n <2^{i+1} n_0$. We have
\begin{equation*}
\phi_{p+3\delta}  [\calC_h \left( 2 n, n \right) ]\ge \phi_{p+3\delta}[\calC_h \left( 4 K_i, K_i \right) ]\stackrel{\eqref{eq:combination}}\ge (1-e^{- \delta 2^{i}})^{5}\ge 1-e^{-c'n}\end{equation*}
for some constant $c'>0$ small enough. This implies for some constant $c>0$, for every $n \ge 1$
\begin{equation*}
\phi_{p+3\delta}  [\calC_h \left( 2 n, n \right) ]\ge 1-e^{-cn}.\end{equation*}
Since $\delta$ is arbitrary, Proposition~\ref{prop:exponential decay2} follows readily. We now can concentrate on the proof of Lemma~\ref{lem:1a}.

\begin{proof}[Proof of Lemma~\ref{lem:1a}] We prove the result by induction on
  $i\ge i_0$.   For $i = i_0$, 
\eqref{eq:combination} implies that
$$\phi_p[\calC_h(2 K_{i_0},n_0)]\ge \phi_p[\calC_h(2n_0,n_0)]^{4\cdot 2^{i_0}}.$$

Then, by considering the translates of $\calC_h(2K_{i_0},n_0)$ by the vector
$(0,(2j-1) n_0)$ with $-2^{i_0-1}< j  \leq 2^{i_0-1}$, we deduce that 
\begin{equation*}\phi_{p}[\sfN(2K_{i_0},K{i_0})\ge 2^{i_0}] \stackrel{\rm (FKG)} \ge
  \phi_p[\calC_h(2n_0,n_0)]^{4\cdot 2^{2i_0}} \stackrel{\eqref{eq:16}}\ge \frac12.
\end{equation*} 

Let us move to the induction step. Let $i\ge i_0$ such  that \eqref{eq:8} holds.
First, by Lemma~\ref{sec:lemGluingManyPaths}, we have 

\begin{equation*}
  \label{eq:15}
  \phi_{p_i}[\sfN(2K_{i+1},K_{i+1})\ge c_i 2^{i+1}]\ge \frac1{2^{10}}.
\end{equation*}
Note that  $ \left( \sfN(2K_{i+1},K_{i+1})-c_{i+1}2^{i+1} \right) \vee 0$ is exactly equal to
$H_{\sfN(2K_{i+1},K_{i+1})<{c_{i+1}2^i}}$. Hence, the above equation implies
  \begin{equation*}
    \label{eq:17}
    \phi_{p_i}[H_{\sfN(2K_{i+1},K_{i+1})<{c_{i+1}2^{i+1}}}]\ge \frac1{2^{10}} (c_i-c_{i+1})2^{i+1}=\frac1{2^{11}{(i+1)}^2} 2^i.
  \end{equation*}
  Lemma~\ref{lem:hamming1} applied to $A=\{\sfN(2K_{i+1},K_{i+1})\ge c_{i+1}2^{i+1}\}$  implies that
\begin{equation*}
\phi_{p_{i+1}}[\sfN(2K_{i+1},K_{i+1}) \ge c_{i+1}2^{i+1}]\ge
1-\exp\big(-\tfrac{(p_{i+1}-p_i)2^{i-9}}{(i+1)^2}\big)=1-\exp\big(-\tfrac{\delta2^{i-9}}{(i+1)^4}\big) \stackrel{\eqref{eq:C1}}\ge \frac12.
\end{equation*}
which concludes the proof.
\end{proof}

\bigbreak

\noindent\textbf{Remarks and comments.}

\medbreak\noindent {\bf 1.} The argument is inspired by a similar yet less powerful argument introduced in \cite{duminil2014phase}.
\medbreak\noindent
{\bf 2.} In \cite{BefDum12}, exponential decay was proved in two steps. First, the cluster of the origin was proved to have finite moments of any order. Then Theorem~5.60 of \cite{Gri06} implied the proof. Note that Theorem~5.60 uses the Domain Markov property and is based on a non-trivial theorem of Kesten \cite{GanKes94}. The argument presented here avoids the use of the domain Markov property and is self-contained.  
\medbreak\noindent
{\bf 3.} Equation \eqref{eq:hamming} is usually stated (see \cite[Theorem~2.56]{Gri06} or \cite{GriPiz97}) in terms of the number $H_A$ of edges that must be switched to open in $\omega$ to be in $A$ (this is the Hamming distance to $A$ in $\{0,1\}^{E_G}$), and reads  
  \begin{align}\label{eq:hamming0}
    \frac{d}{dp}\log(\phi_{G,p}[A])\ge\frac{\phi_{G,p}[H_A]}{p(1-p)},
  \end{align}
which, when integrated between $p-\delta$ and $p$, gives
$$\phi_{G,p-\delta}[A]\le \exp\big(-4\delta\phi_{G,p}[H_A]\big).$$
This inequality is useful to prove that a probability is close to 0, while \eqref{eq:hamming1} is useful to prove that the probability is close to 1.

\section{Proof of Theorem~\ref{main:corollary}}
\label{sec:proof-coroll-refm}
Consider the dual lattice $(\bbZ^2)^*=(\tfrac12,\tfrac12)+\bbZ^2$. Every edge $e$ of $\bbZ^2$ crosses exactly one edge of $(\bbZ^2)^*$ in its middle. We denote this edge by $e^*$. Now, let $\omega^*$ be the configuration on $(\bbZ^2)^*$ defined by $\omega^*_{e^*}=1-\omega_e$.

Consider the measure $\phi^0_{p}$ to be the random-cluster measure with edge-weights $p$ and $q$, and free boundary conditions on $\bbZ^2$. The only properties of $\phi^0_{p}$ that we will use are the following (we refer to \cite{Gri06} for a definition of the free boundary conditions and the two following properties):
\begin{description}
\item[Stochastic domination] For any increasing event $A$, $\phi_{p}^0[A]\le \phi_{p}^1[A]$.
\item[Duality] If $\omega$ is sampled according to $\phi^1_{p}$, then $\omega^*$ is sampled according to the measure $\phi_{p^*}^0$ translated by the vector $(\tfrac12,\tfrac12)$, where
$$\frac{pp^*}{(1-p)(1-p^*)}=q.$$
\end{description}
Let $\calE$ be the event that there is an open path in $\omega$ from left to right in $\llb 0,2n+1\rrb\times\llb 0,2n\rrb$. Also introduce $\calE^*$ be the event that there is an open path in $\omega^*$ from top to bottom in $\llb \tfrac12,2n+\tfrac12\rrb\times\llb -\tfrac12,2n+\tfrac12\rrb$.
Since either $\calE$ or $\calE^*$ occur, we deduce that
\begin{align*}1&=\phi_{p}^1\Big[\calE\Big]+\phi^1_{p}\Big[\calE^*\Big]\stackrel{\text{Duality}}=\phi_{p}^1\Big[\calE\Big]+\phi_{p^*}^0\Big[\calE\Big].\end{align*}
For $p=\sqrt q/(1+\sqrt q)$, $p^*=p$ and therefore the stochastic domination implies that 
$$\phi_{\sqrt q/(1+\sqrt q)}^1\Big[\calC_h(n,n) \Big]\ge \phi_{\sqrt q/(1+\sqrt q)}^1\Big[\calE\Big]\ge \tfrac12.$$
Theorem~\ref{thm:main} implies that for any $p>\sqrt q/(1+\sqrt q)$, there exists $c>0$ such that for every $n\ge1$,
\begin{equation}
\label{eq:abbb}
\phi_{p}^1[\calC_h(2n,n)]\ge 1-e^{-cn}.
\end{equation}
For every $k$, define $\calA_k$ to be $\calC_h(2^{k+1},2^k)$ if $k$ is even, and its image by the rotation of angle $\pi/2$ around the origin if $k$ is odd. If all the $\calA_\ell$ occur simultaneously for every $\ell$ such that $2^\ell\ge n/2$, then $\llb -n,n\rrb^2$ is connected to infinity. Thus, there exists $c'>0$ such that for every $n\ge1$,
$$\phi_{p}^1\Big[\llb-n,n\rrb^2\longleftrightarrow \infty\Big]\stackrel{\rm (FKG)}\ge \prod_{\ell\in\bbN:2^\ell\ge n/2}\phi_{p}^1[\calA_\ell]\stackrel{\eqref{eq:abbb}}\ge  \prod_{\ell\in\bbN:2^\ell\ge n/2} (1-e^{-c2^\ell})\ge 1-e^{-c'n}.$$
This implies that $p_c\le \sqrt q/(1+\sqrt q)$.

Now, fix $p<\sqrt q/(1+\sqrt q)$. Let $x$ with $\|x\|=n$ and let $A_x$ be the event that $x$ is connected to $0$ in $\llb -n,n\rrb^2$. 
We wish to bound the probability of $A_x$. Without loss of generality, we assume that the first coordinate of $x$ equals $n$. Then, define the event $B=\tau_{(-n,0)} A_x\cap A_x$ and $\tilde B$ its symmetric with respect to the $y$ axis. We have that 
$$\phi_{p}^0[B\cap \tilde B]\stackrel{\rm (FKG)}\ge \phi_{p}^0[B]^2\stackrel{\rm (FKG)}\ge \phi_{p}^0[A_x]^4.$$
Now, let $C_1=\tau_{(0,2n)}B\cap\tilde B$ and $C_2$, $C_3$ and $C_4$ the images by the rotations of angles $\pi/2$, $\pi$ and $3\pi/2$ around the origin. We deduce that 
$$\phi_{p}^0[C_1\cap C_2\cap C_3\cap C_4]\stackrel{\rm (FKG)}\ge\phi_{p}^0[B\cap \tilde B]^4\ge \phi_{p}^0[A_x]^{16}.$$
Yet, on $C_1\cap C_2\cap C_3\cap C_4$, we have that $\llb -n,n\rrb^2$ is not connected to infinity in $\omega^*$. Using duality (note that $p^*\ge \sqrt q/(1+\sqrt q)$), the previous inequality implies that
\begin{equation}\label{eq:abcc}\phi_{p}^0[A_x]^{\,16}\le 1-\phi^1_{p^*}[\llb-n,n\rrb^2\leftrightarrow \infty]\stackrel{\eqref{eq:abbb}}\le e^{-cn},\end{equation}
for a constant $c>0$ depending on $p$ and $q$ only. 

Now, if $0$ is connected to $x$, then $0$ must be connected to a vertex on the boundary of the box of size $\|x\|$ {\em inside} the box itself. Hence, \eqref{eq:abcc} implies that for any $x$,
$$\phi_{p}^0[0\leftrightarrow x]\le 8\|x\|e^{-c\|x\|/16}.$$
The previous inequality implies \eqref{eq:abcd} for every $p<\sqrt q/(1+\sqrt q)$ satisfying $\phi_{p}^1=\phi_{p}^0$. This can be extended to every $p<\sqrt q/(1+\sqrt q)$ using the fact that the set of values of $p$ for which $\phi_{p}^1\ne\phi_{p}^0$ is at most countable (see \cite[Theorem~4.58]{Gri06}). Thus \eqref{eq:abcd} holds for every $p<\sqrt q/(1+\sqrt q)$. In particular it shows that $p_c \geq \sqrt q/(1+\sqrt q)$.

\bigbreak

\noindent\textbf{Remarks and comments.}

\medbreak\noindent
{\bf 1.} The inequality $p_c(q)\ge \sqrt q/(1+\sqrt q)$ was already proved using an argument due to Zhang in \cite{Gri06}. However, we did not use this inequality in the present argument.

\paragraph{Acknowledgments} The authors are thankful to Ioan Manolescu for carefully reading the manuscript and his helpful comments. 
This research was supported by the NCCR SwissMAP, the ERC AG COMPASP, and the Swiss NSF.

\bibliographystyle{alpha}
\bibliography{biblicomplete}
\small\begin{flushright}
\textsc{D\'epartement de Math\'ematiques}
  \textsc{Universit\'e de Gen\`eve}
  \textsc{Gen\`eve, Switzerland}
  \textsc{E-mail:} \texttt{hugo.duminil@unige.ch}, \texttt{aran.raoufi@unige.ch}, \texttt{vincent.tassion@unige.ch}
\end{flushright}
\end{document}